\documentclass[11pt,reqno,tbtags]{amsart}

\usepackage[normalem]{ulem}
\usepackage{bbm}
\usepackage{mathrsfs}  
\usepackage{amsthm,amssymb,amsfonts,amsmath}
\usepackage{amscd} 
\usepackage[numbers, sort&compress]{natbib} 
\usepackage{subfigure, graphicx, caption} 
\usepackage{vmargin}
\usepackage{setspace}
\usepackage{paralist}
\usepackage[pagebackref=true]{hyperref}
\usepackage{color}
\usepackage[normalem]{ulem}
\usepackage{enumerate}
\usepackage{enumitem}
\usepackage{stmaryrd} 
\usepackage[refpage,noprefix,intoc]{nomencl} 
\usepackage[normalem]{ulem}

\providecommand{\Z}{}
\providecommand{\N}{}

\providecommand{\T}{}

\renewcommand{\Z}{\mathbb{Z}}
\renewcommand{\N}{{\mathbb N}}

\renewcommand{\T}{\mathbb T}

\newcommand{\E}[1]{{\mathbb E}\left[#1\right]}						
\newcommand{\prob}{\mathbb{P}}

\newcommand{\p}[1]{{\mathbb P}\left(#1\right)}

\newcommand{\I}[1]{{\mathbf 1}_{\left[#1\right]}}
\newcommand{\Inb}[1]{{\mathbf 1}_{#1}}

\DeclareMathSymbol{\leqslant}{\mathalpha}{AMSa}{"36} 
\DeclareMathSymbol{\geqslant}{\mathalpha}{AMSa}{"3E} 
\DeclareMathSymbol{\eset}{\mathalpha}{AMSb}{"3F}     
\renewcommand{\leq}{\;\leqslant\;}                   
\renewcommand{\geq}{\;\geqslant\;}                   

\newcommand{\heap}[2]{\genfrac{}{}{0pt}{}{#1}{#2}}

\newcommand{\cE}{\mathcal{E}}

\newcommand{\eps}{\varepsilon }

\renewcommand{\P}{\mathbb{P}}

\newcommand{\bx}{\mathbf{x}}
\newcommand{\hn}{^{_{(n)}}}
\newcommand{\hnn}{^{(n)}}

\newtheorem{theorem}{Theorem}
\newtheorem{lemma}[theorem]{Lemma}

\newtheorem{remark}[theorem]{Remark}
\newtheorem{corollary}[theorem]{Corollary}

\newtheorem{claim}[theorem]{Claim}

\makenomenclature										
\graphicspath{ {./images/} }
\usepackage{wrapfig}

\DeclareSymbolFont{extraup}{U}{zavm}{m}{n}
\DeclareMathSymbol{\varheart}{\mathalpha}{extraup}{86}
\DeclareMathSymbol{\vardiamond}{\mathalpha}{extraup}{87}

\newcommand{\invisible}[1]{}

\usepackage{hyperref}

\begin{document}

\title[The fewest-big-jumps principle and an application to random graphs]{
The fewest-big-jumps principle and an application to random graphs} 

\
\author[Kerriou]{C\'eline Kerriou}
\address{Universit\"at zu K\"oln, Department of Mathematics and Computer Science, Weyertal 86-90, 50931 Cologne, Germany}
\email{ckerriou@math.uni-koeln.de}
\author[Mörters]{Peter M\"orters}
\address{Universit\"at zu K\"oln, Department of Mathematics and Computer Science, Weyertal 86-90, 50931 Cologne, Germany}
\email{moerters@math.uni-koeln.de}
\address{}
\email{}


\ \\[-15mm]\begin{abstract} 
   We prove a large deviation principle for the sum of $n$ independent heavy-tailed random 
variables, which are subject to a moving 
{truncation}
at location $n$. Conditional 
on the sum being large at scale $n$, we show that a finite number of summands take values near the {truncation} boundary, while the remaining variables still obey the law of large numbers. This generalises the well-known
single-big-jump principle for random variables without {truncation}
to a situation where just the minimal necessary number of jumps occur. 
As an application, we consider a random graph with vertex set given by the lattice points of a torus with sidelength $2N+1$. Every vertex is the centre of a ball with random radius sampled from a heavy-tailed distribution. Oriented edges are drawn from the central vertex to all other vertices in this ball. When this  graph is conditioned on having an exceptionally  large number of edges we use our main result to show that, as $N\to\infty$, the excess outdegrees condense in a fixed, finite number of randomly scattered vertices of macroscopic outdegree. By contrast, no condensation occurs for the indegrees of the vertices, which all remain microscopic in size.
\end{abstract}

\maketitle
\ \\[-18mm]

\section{Introduction}
There are two principal paradigms for the upper large deviations of the sum of independent (or mildly dependent) and identically distributed random variables. These are on the one hand the 
\emph{paradigm of collaboration} whereby all variables contribute equally to the sum by a change of measure applied to all independent terms. This large deviation strategy typically prevails if the terms are light tailed at infinity and the  cost of individual large terms would be overwhelming. On the other hand there is the \emph{paradigm of the single big jump}, see for example~\cite{10.2307/25450633, ASMUSSEN1996103, foss2005principle, PhysRevE.100.012108, NWZ2022}, by which a single summand provides a sufficiently large value so that it suffices that all other terms show typical behaviour. In statistical mechanics this paradigm is related to spontaneous symmetry breaking, i.e. a single term is chosen at random to provide the large jump, and condensation, i.e.\ a single site provides a macroscopic share of the total mass. This large deviation strategy typically prevails in the case of sums of heavy-tailed random variables.
\smallskip

In this paper we look at a situation where the random variables are heavy-tailed but there is a {truncation} at a variable threshold. In this case a single big jump may not suffice, but rather a finite number of summands are required to provide the large deviation. We  prove a precise large deviation result showing that the optimal strategy in this case is to pick the minimal number of sites necessary for the large deviation and share the excess mass among these sites. This result, Theorem~\ref{thm:main_theorem},  underpins the novel \emph{paradigm of the minimal number of jumps}, or fewest-big-jumps principle at the heart of this paper. 
\smallskip

Our general result is motivated by a large deviation principle for the number of edges in  a finite geometric random graph model as the number of vertices goes to infinity, see Theorem~\ref{thm:graph_application}. This model has oriented edges, so that there are two ways to represent the number of edges in the graph as the sum of identically distributed random variables: either as the sum of the indegrees  or as the sum of the outdegrees over all vertices in the graph. We show that from the former point of view we see the paradigm of collaboration, but from the latter we see the paradigm of the minimal number of jumps. Our proof of this insightful result uses Theorem~\ref{thm:main_theorem} applied to the outdegrees. We conclude the paper with some comments on further large deviation results for the number of edges in geometric random graphs.
\pagebreak[3]

\section{The `fewest-big-jumps' principle}

For our main result we consider a triangular scheme of nonnegative random variables
 where the random variables $(W_k^{_{(n)}} \colon 1\leq k\leq n)$ in the $n$th row are
independent and identically distributed to some $W^{{(n)}}$ for which we make the following assumptions:

\begin{enumerate}[wide, label=\textbf{Assumption~\arabic*.},ref=\arabic*,leftmargin=*, labelindent=0pt, labelwidth=!]
    \item \label{eq:expectation_conv_assumption}
    We have $0\leq W^{{(n)}}\leq n$ and there exists $\mu\geq 0$ such that $$\mu_n:=\mathbb E W^{{(n)}} \to \mu.$$
    \item \label{eq:distr_assumption}  There exist $\alpha>1$, a regularly varying sequence $(f(n))_n$ with index $-\alpha$,
    a measure $\nu$ on $(0,1]$, which is finite on intervals bounded away from zero and absolutely continuous on $(0,1)$ with a continuous
    density $h\colon (0,1) \to [0,\infty)$, and a sequence $(\eta_n)_n$ with
    $\eta_n \downarrow0$ such that, for all $\eps>0$ whenever $\eps \leq a_n< b_n$ with $b_n \leq 1-\eta_n$ or $b_n=1$  and with
    $\eta_n\leq b_n-a_n\to0$, we have
    \begin{equation*}
        \prob(a_n n <  W^{{(n)}} \leq b_n n) =  ( 1  +o^*(1)) \, f(n)\, \nu(a_n,b_n],
    \end{equation*}
    where $o^*(1)$ is a function going to zero such that, for all fixed $\eps>0$ and $\Delta_n\downarrow0$, the convergence is
    uniform over all eligible choices of $a_n, b_n$ with $b_n-a_n\leq\Delta_n$.
\end{enumerate}

From Assumption~\ref{eq:distr_assumption} we obtain, for every $0<\eps<1$, that
\begin{equation}\label{eq:lp_bound_assumption}
        \prob(W^{{(n)}}>\eps n ) =  ( 1  +o(1)) \, f(n)\, \nu(\eps,1].
    \end{equation}
This follows by partitioning $(\eps,1]$ into 
exactly $(1-\eps)/\Delta_n$ disjoint intervals of equal length $\Delta_n$ chosen such that $\eta_n\leq \Delta_n\to0$, and applying Assumption~\ref{eq:distr_assumption} to these intervals.
\smallskip

\begin{remark}
Assumption~\ref{eq:distr_assumption} roughly says that, for every $0<\eps<1$, conditionally on $W^{{(n)}}> \eps n$, the random variable $W^{{(n)}}/{n}$ converges in a certain strong sense to the distribution $\nu|_{(\eps,1]}/\nu((\eps,1])$. The measure $\nu$ therefore describes the shape of the truncation on the scale~$n$.
We allow $\nu$ to have an atom at one in order to cover examples with a hard cut-off, while a continuous density of $\nu$ is needed elsewhere to allow the explicit description of constants in our results. The sequence $(\eta_n)_n$ is introduced to weaken the assumption with the aim to cover more examples, in particular in the context of random graphs.\end{remark}

\noindent
{\bf Examples:}
Our examples are based on a random variable $W\geq 0$ with regularly varying tail function $G(x):=\mathbb P(W>x)$. More precisely, we assume that $G(x)=x^{-\alpha} \ell(x)$ with
$\alpha>1$ and the slowly varying part $\ell$ 
satisfies, for all $0<\eps \leq a_n< b_n$ with $b_n-a_n\to0$, 
\begin{align}\label{eq:slowly_var_condition}
    \frac{\ell(b_n n)-\ell(a_n n)}{\ell(n) \,\frac{b_n-a_n}{b_n}} \to 0  \text{ as  } n\to\infty,
\end{align}
where for all fixed $\eps>0$ and $\Delta_n\downarrow0$, the convergence is uniform over all choices of $\eps\le a_n< b_n$ with $b_n-a_n\leq\Delta_n$.
This condition holds, e.g., if $\ell$ is a constant,  logarithm, or the reciprocal or iterate of a logarithm.
\begin{enumerate}
[wide, label={\bf(\alph*)}, ref=\alph*,leftmargin=*, labelindent=0pt, labelwidth=!]
    \item\label{ex:a} The random variables $W^{{(n)}}:=W\wedge n$ satisfy Assumption~\ref{eq:expectation_conv_assumption}. For all $n$ with $b_n<1$ we have
    \begin{align*}
         \mathbb P(a_n n & < W \wedge n \leq b_n n) = \p{W > a_n n} - \p{W> b_n n} 
         = G(a_n n) - G(b_n n)\\
         & = n^{-\alpha} \big( (a_n^{-\alpha} - b_n^{-\alpha}) \ell(a_n n)+ (b_n^{-\alpha} (\ell(a_n n)-\ell(b_n n) )  \big) \\
         & = n^{-\alpha} \big( (a_n^{-\alpha} - b_n^{-\alpha}) (1+o^*(1))  \ell(n)  + o^*(1) \ell(n) b_n^{-\alpha-1}(b_n-a_n)   \big) \\
        & = (1+o^*(1)) n^{-\alpha} \ell(n) \, (a_n^{-\alpha} - b_n^{-\alpha}),
    \end{align*}
    and for all $n$ with $b_n=1$ that
\begin{align*}
         \mathbb P(a_n n  < W \wedge n \leq b_n n) = \p{W > a_n n} 
         = G(a_n n)  = n^{-\alpha} a_n^{-\alpha} \ell(a_n n) 
        = (1+o^*(1)) n^{-\alpha} \ell(n).
 \end{align*}    
Hence Assumption~\ref{eq:distr_assumption} holds with $f=G$, arbitrary $\eta_n\downarrow0$, and the measure $\nu$ having density $\alpha x^{-\alpha-1}$
    on $(0,1)$ and an atom of size one at one.\smallskip
     \item\label{ex:b} Let $\phi_n\colon (0,\infty)\to(0,n)$ be a family of increasing, surjective functions.  Suppose $(\phi_n(W))_n$ is $L^p$-bounded for $1\leq p<\alpha$, $\E{ \phi_n(W)}$ converges and, for some continuously differentiable, 
     increasing  function 
     $g\colon (0,1) \to (0,\infty)$ with \mbox{$g(a)\to\infty$} as $a\uparrow1$, 
     and for some $\eta_n\downarrow0$ and all $0<\eps \leq a_n< b_n \leq 1$ with    $b_n-a_n\geq {\eta_n}$
     we have
\begin{equation}\label{koc}
\sup_{a\in(0,1)} \Big|\frac{\phi^{-1}_n(a n)}n-g(a)\Big| \leq  \eta_n^2
\text{ and } \frac{\ell(\phi^{-1}_n(a_n n))}{\ell(n)}=1+o^*(1).
\end{equation}
Note that by \cite[Theorem~1.2.1]{bingham_goldie_teugels_1987} the second condition in \eqref{koc} need
only be checked for $a_n\to 1$ and ${1-a_n}\geq{\eta_n}$.
     We have for $b_n<1$ that
    \begin{align*}
        \mathbb P(a_n n & < \phi_n(W)\leq b_n n)  = \p{W > \phi_n^{-1}(a_n n)} - \p{W>\phi_n^{-1}(b_n n)} \\
        & = \ell(\phi_n^{-1}(a_n n))  \big( (\phi_n^{-1}(a_n n))^{-\alpha}
        - (\phi_n^{-1}(b_n n))^{-\alpha}\big) \\
& \qquad          +  (\phi_n^{-1}(b_n n))^{-\alpha} \big( \ell(\phi_n^{-1}(a_n n)) - \ell(\phi_n^{-1}(b_n n)) \big) \\
         & = (1+o^*(1)) {\ell(\phi_n^{-1}(a_n n))}   \big( (ng(a_n))^{-\alpha}
        - (ng(b_n))^{-\alpha}\big)  \\
& \qquad          +  o^*(1) n^{-\alpha} g(b_n)^{-\alpha-1}   \big( \ell(n) { (g(a_n) - g(b_n))} \big)  \\
        & = (1+o^*(1))  \ell(n)  n^{-\alpha}
        \, (g(a_n)^{-\alpha}-g(b_n)^{-\alpha}),
    \end{align*}
    and, in the case $b_n=1$,
\begin{align*}
         \mathbb P(a_n n < \phi_n(W)) & = \p{W > \phi_n^{-1}(a_n n)} 
      = (\phi_n^{-1}(a_n n))^{-\alpha} 
     {\ell(\phi_n^{-1}(a_n n))}  \\
        & = (1+o^*(1)) n^{-\alpha} \ell(n)  
        \,  g(a_n)^{-\alpha}.
    \end{align*}
    Hence $W^{{(n)}}:=\phi_n(W)$ satisfies our assumptions 
    with $f=G$ and $\nu$ absolutely continuous
    with density \smash{$h(x) = \alpha g(x)^{-\alpha-1}g'(x)$}. 
    \smallskip

    A possible choice of such a {truncation} family would be \smash{$\phi_n(x)=n(1-e^{-x/n})$}, in which case \smash{$\phi_n^{-1}(x) = -n\log(1-x/n)$}
    and $g(x)=-\log(1-x)$. Then the measure $\nu$ is absolutely continuous with density
    $h(x)=\alpha (1-x)^{-1}\log(\frac1{1-x})^{-\alpha-1}$. Given $\ell$ the second condition in \eqref{koc}
    becomes a condition on the sequence~$(\eta_n)_n$.\smallskip
    \item\label{ex:b2}
    The random variables $W^{{(n)}}$ defined as
    $W$ conditioned on the event $\{W\leq n\}$ satisfy  Assumptions~\ref{eq:expectation_conv_assumption}. Moreover, for all $0<\eps \leq a_n< b_n \leq 1$,  we have 
    $$\prob(a_n n <  W^{{(n)}} \leq b_n n) = \frac{G(a_n n)-G(b_n n)}{1-G(n)} =(1+o^*(1)) G(n) (a_n^{-\alpha}- b_n^{-\alpha}),$$ so that Assumption~\ref{eq:distr_assumption} holds with $f=G$, an arbitrary sequence  $\eta_n\downarrow0$, and $\nu$ absolutely continuous  
    with density
    $h(x)=\alpha x^{-\alpha-1}$.\pagebreak[3]
    
    \item\label{ex:c}  Let $W^{{(n)}}$ be the number of points in
    $B(0,W^{\frac1d}) \cap [-N,N]^d \cap \Z^d$ where 
    $n=(2N+1)^d$. Then we show in Section~\ref{sec:randomgraphmodel} that
     $W^{{(n)}}$ satisfies our assumptions.
\end{enumerate}
 
\begin{theorem}\label{thm:main_theorem}
Suppose that $W_1^{_{(n)}},\ldots, W_n^{_{(n)}}$ are independent and identically distributed random variables satisfying Assumptions~\ref{eq:expectation_conv_assumption} and~\ref{eq:distr_assumption}, and let $$S_n=W_1^{_{(n)}}+\ldots +W_n^{_{(n)}}.$$
Let $\rho>0$ be non-integer and $k = k(\rho) \in\mathbb N$  
such that
$$k-1<\rho<k.$$
There exists $\delta_n\downarrow0$ such that, for any 
$\rho_1(n)\uparrow\rho$ and $\rho_2(n)\downarrow \rho$ with
$\rho_2(n)-\rho_1(n)\gg \delta_n$, we have
$$\prob\big(n(\rho_1(n)+\mu)\leq S_n \leq n(\rho_2(n)+\mu)\big) =
(K_{\rho}+o(1)) {n \choose k} (\rho_2(n)-\rho_1(n))  f(n)^{k},$$
where  $K_{\rho} := h(\rho)$ if $0<\rho<1$ and for $\rho>1$,
$$K_{\rho}:=
\int \nu(dx_1)\cdots 
         \int \nu(dx_{k-1}) \, \frac{k}{k-\sum_{i=1}^{k-1}1_{x_i=1}} \, 
\, h\Big(\rho-\sum_{i=1}^{k-1}x_i\Big),$$
where $h$ is extended by zero to  the entire real line 
and $K_\rho$ is finite.
\end{theorem}

\begin{remark}
\label{tail}
By a variation of the argument we also obtain that, under the assumptions of the theorem, we have for \smash{$K'_{\rho}:=
\int \nu(dx_1)\cdots 
         \int \nu(dx_{k})
         \mathbf 1_{[\sum_{i=1}^k x_i \geq \rho]}$}, that 
$\prob(S_n \geq n(\rho+\mu)) =
(K'_{\rho}+o(1)) {n \choose k}   f(n)^{k}$, see Section~\ref{sec:proof_main_theorem} for a sketch of the modifications needed in the proof.
\end{remark}
\begin{remark}\label{integer}
Observe that the order in $n$ of the probability of the large deviation event in Theorem~\ref{thm:main_theorem} and  Remark~\ref{tail} changes discontinuously at integer values of $\rho$. If $\rho$ itself is an integer even the order of the tail probability depends on the exact distribution of the summands in $S_n$ through the convergence rates in the law of large numbers and the tail at $n$. To illustrate this dependence consider
$W^{(n)}=\phi_n(W)$ as in our example $(b)$ with $G(x)=x^{-\alpha}$ for $x\ge1$. In the case of the cut-off function
$\phi_n(x)=x \wedge n$ the tail probability 
is of order 
$n^{-(\alpha+1)k}$ because 
to achieve $S_n \geq n(k+\mu)$
it suffices to raise the value of $W$ for $k$ instances to $n$ or larger and have the rest following the law of large numbers, 
while in the case $\phi_n(x)=\frac{nx}{n+x}$
this does not suffice and hence the tail probability is of strictly smaller order than~$n^{-(\alpha+1)k}$.
\end{remark}
The proof of Theorem~\ref{thm:main_theorem}, given in Section~\ref{sec:proof_main_theorem}, also reveals explicitly how the summands achieve
the large deviation of the sum $S_n$, namely by making exactly the minimal number $k$ of jumps necessary to increase the sum by $\rho n$. The following corollary makes this observation precise. For the rest of the paper, define for $\rho_1(n), \rho_2(n)$ as in Theorem~\ref{thm:main_theorem},
\begin{align}\label{def:interval_in}
    I_n := I_n(\rho_1(n), \rho_2(n)) := [n(\rho_1(n) + \mu), n(\rho_2(n) + \mu)].
\end{align}
\begin{corollary}\label{cor:dominating_event} 
There exists $\gamma_n
\downarrow 0$ such that for  sufficiently small $\eps>0$ we have that
\begin{align*}
 \prob\Big(  & \text{ there exist exactly } k \text{  indices } 1\leq i_1,\ldots ,i_k \leq n
    \text{ such that}\\
    & \phantom{lu} W_{i_1}\hn> \eps  n,\ldots , W_{i_k}\hn > \eps  n
    \text{ and } \Big|\sum_{j=1}^k W_{i_j}\hn   - \rho n\Big| \leq \gamma_n n \text{ and }\\
     & \phantom{lu} W_{j}\hn \leq \eps  n \, \forall \,  j\not=i_1,\ldots,i_k
    \text{ and } \Big|\sum_{\heap{j=1}{j\not=i_1,\ldots,i_k}}^n W_j\hn   - \mu n\Big| \leq \gamma_n n \, 
    \Big| \, S_n \in I_n \Big) \rightarrow 1,
\end{align*}
as $n\rightarrow \infty$. 
\end{corollary}

Before moving to an application 
to random geometric graphs, we comment on the possible size 
of  the sequence $(\delta_n)_n$ in Theorem~\ref{thm:main_theorem}.

\begin{remark}In Section~\ref{sec:proof_main_theorem} we show that there exists 
$\zeta_n\downarrow 0$ such that $S_n:=W_1^{_{(n)}}+\cdots +W_n^{_{(n)}}$
satisfies
\begin{equation}
\prob\big(| S_n-n \mu| > \zeta_n n\big) \leq \zeta_n.
\label{eq:lln_wn}\end{equation}
Under suitable assumptions, explicit values for $(\zeta_n)_n$ can be derived from central or stable limit theorems for  triangular schemes.
The sequence $(\delta_n)_n$ in Theorem~\ref{thm:main_theorem} can then be chosen as $\delta_n\gg \zeta_n \vee \eta_n$ with $(\eta_n)_n$ as in  Assumption~\ref{eq:distr_assumption} and $(\zeta_n)_n$ as in~\eqref{eq:lln_wn}. We additionally require $\zeta_n \gg n^{-1}$ and $\zeta_n \gg n f(n)$ for $n$~large, which is typically weaker than the requirement of~\eqref{eq:lln_wn}. Hence $(\delta_n)_n$ can be chosen independent of $\rho$, but note that the convergence in Theorem~\ref{thm:main_theorem} is not uniform over $\rho\in(k-1,k)$.
\end{remark}


\section{The random graph model}\label{sec:randomgraphmodel}

We now look at a large deviation problem for the number of edges in
a geometric random graph model $G = (V,E)$. The vertices of this graph are the 
$n=(2N+1)^d$ lattice points on $\T_N:= [-N,N]^d$ equipped with 
the torus metric~$D$.
Every vertex $v\in V$ is equipped with a random radius $R_v\geq0$ where $(R_v \colon v\in V)$ are i.i.d.\ random variables with regularly varying tail function $G(x):=\mathbb P(R_v>x)$. We assume that $G(x)=x^{-\beta} \ell(x)$ with
$\beta>d$ and the slowly varying part $\ell$ 
satisfies condition \eqref{eq:slowly_var_condition}.
We draw an oriented edge from each vertex $v$ to all vertices $w\in
B(v,R_v)\setminus\{v\}$, where $B(v,r):=\{ w\in \T_N \colon D(v,w)<r\}$ is the open ball centred in~$v$ of radius~$r$. \smallskip

We denote by
$\mathsf{out}(v)$ the out-, and by $\mathsf{in}(v)$ the in-degree of a vertex $v$. Define~by
$$\varrho(n):= \frac1n\sum_{v\in \T_N} \mathsf{out}(v)
= \frac1n\sum_{v\in \T_N} \mathsf{in}(v)$$
the edge density and by
$\displaystyle\mu:=\lim_{n\to\infty} \mathbb E \varrho(n)$ the asymptotic mean edge density.

The following theorem is our main application of Theorem~\ref{thm:main_theorem}. It shows that when the graph $G$ is conditioned to have an edge density larger than $\mu$, the excess outdegrees condense in a fixed, finite number of randomly scattered vertices whereas the indegrees all remain microscopic in size.
\begin{theorem}\label{thm:graph_application}
Let $\rho>0$ be non-integer and $k\in\N$ such that $k-1<\rho<k$.
There exist $\delta_n\downarrow0$ and a constant $K_\rho>0$ such that,  
for any sequences
$\rho_1(n)\uparrow\rho$ and $\rho_2(n)\downarrow \rho$ with
$\rho_2(n)-\rho_1(n)\gg \delta_n$, we have 
$$\prob\big( \mu+\rho_1(n) \leq \varrho(n) \leq \mu+\rho_2(n)\big)
= (K_{\rho}+o(1)) {n \choose k} (\rho_2(n)-\rho_1(n))  \, G(n)^{k/d}.$$
Given the event $\{\mu+\rho_1(n) \leq \varrho(n) \leq \mu+\rho_2(n)\}$ we have that
\begin{enumerate}[label=(\alph*), ref=\alph*]
    \item \label{eq:cond_total_outdeg}
    There exists a set $\mathscr{V}$ of exactly $k$ vertices $v_1,\ldots, v_k\in\T_N$ such that $$\displaystyle
    \lim_{n\to\infty} \frac{1}{n}  \sum_{i=1}^k \mathsf{out}(v_i) = \rho  \qquad \mbox{ in probability. }$$
 \item \label{eq:cond_outdeg} 
 Every $v\in\mathscr V$ has a macroscopic share of the total {\bf outdegree},~i.e. there exists $\eps>0$
 such that 
    $$\min_{v\in \mathscr V} \frac{\mathsf{out}(v)}{n} \geq \eps  \qquad \mbox{ with high probability,}$$
whereas every $v\not\in\mathscr V$ makes only a microscopic contribution, i.e. $$\displaystyle \lim_{n\to\infty} \max_{v\not\in\mathscr V} \frac{\mathsf{out}(v)}n = 0 \qquad \mbox{ in probability.}$$
    \item \label{eq:no_cond_indeg}No vertex makes a macroscopic contribution to the total {\bf indegree}, i.e. $$\displaystyle \lim_{n\to\infty} \max_{v\in V} \frac{\mathsf{in}(v)}n= 0\qquad \mbox{ in probability.}$$
\end{enumerate}
\end{theorem}

\begin{proof} For every vertex $v\in V$, define $W\hn_v$ to be the outdegree of $v$, that is, $W\hn_v$ is the number of lattice points in 
{$(B(v,R_v)\setminus\{v\})\cap \T_N$}. 
Then $(W\hn_v: v\in V)$ are independent, identically distributed and $\varrho(n) = \sum_{v\in V} W\hn_v /n$. We claim that $W\hnn_v$ satisfies Assumptions~\ref{eq:expectation_conv_assumption} and ~\ref{eq:distr_assumption} 
with $\alpha=\beta/d$ and $f(n)=G(n)^{1/d}$. The first equality of Theorem~\ref{thm:graph_application} then follows directly by Theorem~\ref{thm:main_theorem}, and (\ref{eq:cond_total_outdeg}) and (\ref{eq:cond_outdeg}) follow from Corollary~\ref{cor:dominating_event}. To see why  (\ref{eq:no_cond_indeg}) holds, assume that the indegree of some vertex $v\in V$ is of order $n$. Then $v$ is contained in at least order $n^{(d-1)/d}$ balls of volume of order $n$. In other words, there are at least order $n^{(d-1)/d}$ vertices with outdegree of order $n$, which 
contradicts~(\ref{eq:cond_outdeg}).

Assumption~\ref{eq:expectation_conv_assumption} follows from the fact that $\smash{(W\hnn)_n}$ is increasing to
the number of lattice points in {$B(0,R_0)\setminus\{0\}$}.
We are  left to show that $W\hnn$ satisfies Assumption~\ref{eq:distr_assumption} with $\alpha=\beta/d$ and $f(n)=G(n)^{1/d}$. We do so by bounding $W\hnn$ {from below, resp.\ above,} by the number of unit hypercubes (centred in lattice points) that are contained in, 
resp.~intersect, $(B(0,R_0)\setminus\{0\})\cap \T_N$. Define the function $g: [0,\infty) \rightarrow [0,1]$ by 
$g(r) := \text{Vol}(B(0, \tfrac12\sqrt{d}r) \cap \left[-1/2, 1/2\right]^d).$
Then $g(r) = 1$ for $r\geq 1$. Note that $g'(r)$ corresponds to the {$(d-1)$-dimensional Hausdorff measure of the sphere $\partial B(0,\frac12 \sqrt{d}r)$ intersected with the set $[-1/2, 1/2]^d$}. Therefore $g$ is strictly increasing and continuously differentiable on $(0, 1)$ and we have $g'(x) >0$ for all $x \in (0, 1)$. Moreover,
$$\lim_{x\downarrow 0} g'(x)=0
 \text{ and } 
\lim_{x\uparrow 1} g'(x)=0.$$
Hence $g^{-1}\colon(0,1) \rightarrow (0, 1)$ is continuously differentiable and strictly increasing. Observe that with $g^{-1}(1)=1$, for all $0<a\leq 1$ and $0<\eps< a$, we get that 
\begin{equation}\label{eq:diff_g_inv}
\big|g^{-1}(a)- g^{-1}(a-\eps)\big|
\leq \eps^{1/d}.
\end{equation} 
To see this first for $a=1$ note that $g(1) - g(1-\eps)$ is the volume of the intersection $S_\eps$ of $[-1/2, 1/2]^d$ and the complement of the ball $B(0, (1-\eps)\sqrt{d}/2)$. Let $y=\eps \sqrt{d}/2$ be the distance between a corner of $[-1/2, 1/2]^d$ and the ball. At each corner we can fit an open  cube with diagonal $y$ into $S_\eps$, as illustrated in Figure~\ref{fig:corner_graph}. As $[-1/2, 1/2]^d$ has $2^d$ corners  we can bound $g(1) - g(1-\eps)$ from below by $2^d (y/\sqrt{d})^d = \eps^d$. Since $g^{-1}(1) = 1$ and $g^{-1}$ is strictly increasing, we have that $1-\eps\leq g^{-1}(1-\eps^{d})$ and we can conclude that (\ref{eq:diff_g_inv}) holds for $a=1$. For $0<x<1$ an analogous argument gives 
$x-\eps\le  g^{-1}(g(x)-\eps^d)$ and plugging $x=g^{-1}(a)$ implies the result for 
any~$0<a<1$.
\smallskip

We can express the volume of the intersection of a ball of radius $r>0$ and $\T_N$ in terms of $g$ as 
\begin{align*}
    \text{Vol}(B(0,r) \cap \T_N) & = (2N)^d g\big(\tfrac{r}{\sqrt{d}N}\big).
\end{align*}

\begin{figure}[htp]
    \hspace*{\fill}%
    \begin{minipage}[t]{0.45\textwidth}
        \centering
        \captionsetup{width=\linewidth}
        \vspace{0pt}
    \includegraphics[width=0.94\linewidth]{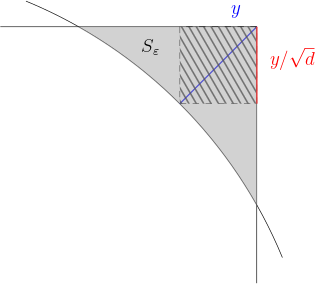}
        \caption{The figure illustrates how an open cube of sidelength $y/\sqrt{d}$ can be fit into $S_{\eps}$.}
        \label{fig:corner_graph}
        \end{minipage}%
        \hfill
    \begin{minipage}[t]{0.45\textwidth}
        \centering
        \captionsetup{width=\linewidth}
        \vspace{0pt}
        \includegraphics[width=.9\linewidth]{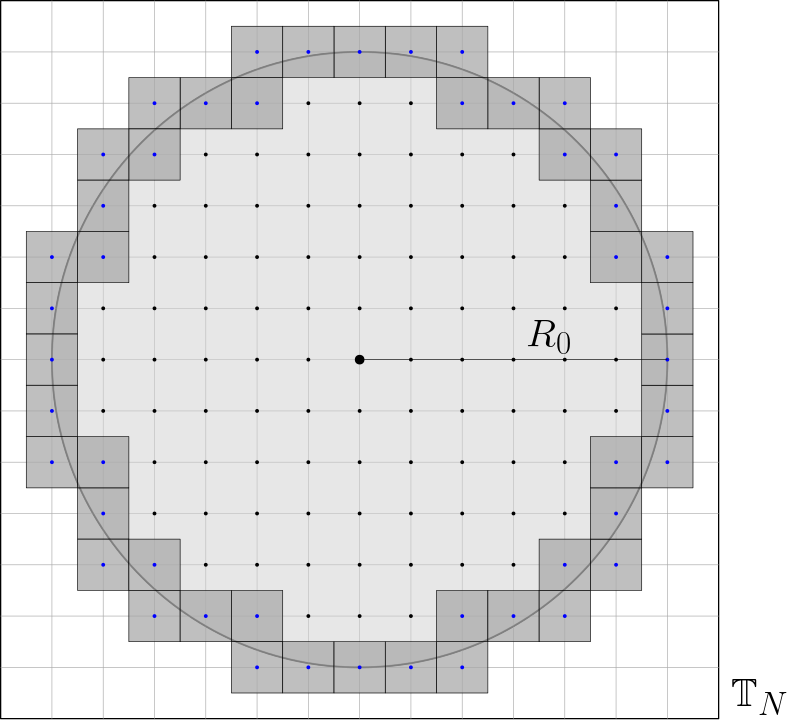}
        \caption{The figure illustrates how $W\hnn$ can be bounded from above by the volume of 
        $B(0,R_0)$ plus the number of unit cubes in $\T_N$ that intersect the boundary of the ball. A lower bound follows by subtracting the number of unit cubes in $\T_N$ intersecting the boundary of the ball from its volume.}
         \label{fig:approx_unit_cubes}
    \end{minipage}%
    \hspace*{\fill}
    \vspace{-3mm}
\end{figure} 

\pagebreak[3]

Note that the boundary of $B(0,r)$ intersects a constant multiple of $N^{d-1}$ unit hypercubes centred in lattice points of~$\T_N$. We bound $W\hnn$ from above by the smallest number of unit hypercubes needed to cover $B(0,R_0)\cap \T_N$. This in turn can be bounded by the volume of $B(0,R_0) \cap \T_N$ summed with the number of unit hypercubes in $\T_N$ intersecting its boundary. A lower bound follows analogously. An illustration appears in Figure~\ref{fig:approx_unit_cubes}. 
Thus it holds that for some constant $c_d>1$,
\begin{align*}
    (2N)^dg\big(\tfrac{R_0}{\sqrt{d}N}\big) - c_d N^{d-1}\leq W\hnn \leq (2N)^dg\big(\tfrac{R_0}{\sqrt{d}N}\big) + c_d  N^{d-1}.
\end{align*}
Recalling that $n = (2N+1)^d$, it follows that for all $\eps>0$ and $\eps\leq a_n < b_n \leq 1$ there exists some constant $c_d'>0$ such that for sufficiently large $n$,
\begin{align*}   \mathbb P(W\hnn > a_n n) & \leq \p{g(\tfrac{R_0}{\sqrt{d} N}) > a_n - c'_dN^{-1}}
     = \p{R_0 > \sqrt{d}N g^{-1}(a_n - c_d'N^{-1})}\\
    & = G\big(\sqrt{d}N g^{-1}(a_n - c_d'N^{-1})\big),
\end{align*}
and similarly, for $1-b_n \gg N^{-1}$, we have 
 $\mathbb P(W\hnn > b_n n) \geq G(\sqrt{d}N g^{-1}(b_n + c_d'N^{-1}))$.
By our assumption on $G$ we have, for all $\eps > 0$ and $\eps\leq c_n< d_n\leq1$ with $d_n-c_n\to0$,
\begin{align}\label{difference}
    G(c_n N) - G(d_n N) & = N^{-\beta}((c_n^{-\beta} - d_n^{-\beta})\ell(c_nN) + d_n^{-\beta}(\ell(c_nN) - \ell(d_nN)))\notag\\
    & =  N^{-\beta} ((c_n^{-\beta} - d_n^{-\beta})(1+o^*(1))\ell(N) + o^*(1)\ell(N)d_n^{-\beta -1}(c_n - d_n)\notag\\
    & = (1+o^*(1)) N^{-\beta} \ell(N) \,  (c_n^{-\beta}-d_n^{-\beta}).
\end{align}

\noindent
For $1-b_n \geq \eta_n \gg N^{-1}$, by (\ref{eq:diff_g_inv}), we can write $g^{-1}(a_n - c_d' N^{-1}) = g^{-1}(a_n) - \delta_n$ and
$g^{-1}(b_n + c_d' N^{-1}) = g^{-1}(b_n) + \delta'_n$
for some $0\leq \delta_n, \delta'_n \leq (c_d'N^{-1})^{1/d}$. We infer by (\ref{difference}) that,  for some constant $c_d''>0$, we have that
\begin{align*}
     & \mathbb P(a_nn < W\hnn \leq b_nn) \leq 
  G\big(\sqrt{d}N g^{-1}(a_n - c_d'N^{-1})\big)-   
  G\big(\sqrt{d}N g^{-1}(b_n + c_d'N^{-1})\big)\\
  & \qquad = (1+o^*(1)) N^{-\beta} \ell(N) \,   {d}^{-\beta/2} \,
   (g^{-1}(a_n - c_d'N^{-1})^{-\beta}-g^{-1}(b_n + c_d'N^{-1})^{-\beta})\\
  & \qquad \leq (1+o^*(1))G(N) \,  
  {d}^{-\beta/2} \, \big(g^{-1}(a_n)^{-\beta}-g^{-1}(b_n)^{-\beta} + c_d''N^{-1/d}\big).
\end{align*}
If $b_n-a_n \geq \eta_n\gg N^{-1/d}$, then 
$$g^{-1}(a_n)^{-\beta}-g^{-1}(b_n)^{-\beta}+ c_d''N^{-1/d}
=(1+o^*(1))\int_{a_n}^{b_n} 
-((g^{-1})^{-\beta})'(x) \, dx.$$
A matching lower bound on $\mathbb P(a_nn < W\hnn \leq b_nn)$ follows analogously. In the case $b_n=1$, we get
\begin{align*}
   \mathbb P(W\hnn & > a_nn)    \leq G\big(\sqrt{d}N g^{-1}(a_n - c_d'N^{-1})\big)\\
    & =(1+o^*(1))  G(N) \big(\sqrt{d} g^{-1}(a_n - c_d'N^{-1})\big)^{-\beta}
     = (1+o^*(1))  G(N) d^{-\beta/2},
\end{align*}
where the second equality holds as $a_n - c'dN^{-1}$ tends to one and $g^{-1}(1) = 1$. In the same manner, we deduce a matching lower bound. Now, define the function $h :(0,1)\rightarrow (0,\infty)$ by
   $h(x) := - d^{-\beta/2} ((g^{-1})^{-\beta})'(x)$, for $x\in(0,1)$.
Note that $h$ is continuous and integrable on intervals bounded from zero. We conclude that $W\hnn$ satisfies Assumption~\ref{eq:distr_assumption} with \smash{$\eta_n \gg N^{-1/d} \sim 2^{1/d}n^{-1/d^2}$}, 
$f(n)=G(N)\sim 2^{-\beta}G(n^{1/d})$
and a measure $\nu$ with density $h$ on $(0,1)$ and an atom of size $d^{-\beta/2}$ at one.
\end{proof}

\section{Proof of Theorem~\ref{thm:main_theorem}}\label{sec:proof_main_theorem}
In Theorem~\ref{thm:main_theorem} we claim that the large value of $S_n$ is achieved by increasing exactly~$k$ of the summands $W_1^{_{(n)}},\ldots,W_n^{_{(n)}}$ to size of order $n$, so that together they give an extra contribution of $\rho n$. We start the proof by Lemma~\ref{lem:sk_approximation}, which is a useful tool that in particular allows us to calculate the probability 
of this principal strategy. The proof is then completed by four claims.  Claim~\ref{claim:new_bound_e121k}  shows that the principal strategy is successful, i.e. it leads to the desired value of $S_n$. We classify possible alternative strategies into three categories:
\begin{itemize}
    \item[--] exactly $k$ summands are of order $n$ but their joint contribution is not $\rho n$,
    \item[--] more than $k$ summands are of order $n$,
    \item[--] fewer than $k$ summands are of order $n$.
\end{itemize}\pagebreak[2]
We show in Claims~\ref{claim:bound_estar21k}--\ref{claim:bound_e4j} that the probability of each of the three above strategies is of strictly smaller order than that of the principal strategy.
\smallskip

Similar methods of proof have been used
to show big jump principles, see for  example \cite{Nagaev79, Janson2011, CD19}, and \cite{RBZ19, CBRZ19} for a different approach. Our proof is modelled on the ideas of  \cite[Theorem 19.34]{Janson2011} for balls-in-boxes models and random trees,
which has the combined advantage of giving a very precise result and including information about the 
geometry of the conditional sample as given in our Corollary~\ref{cor:dominating_event}.
We generalise Janson's technique by removing restrictions] on the discrete nature of the random variables, dealing with variables that may be continuous or discrete, possibly with a complex lattice structure,  and of course by incorporating the {truncation} effect. For $k>1$ the latter leads to significant additional difficulties in the argument as extra randomness has to be controlled coming from different possible distributions of the mass among the $k$~condensates. 
\smallskip

We first argue that $K_\rho<\infty$. This is obvious if $k=1$. Otherwise, because
$\rho>k-1$ there exists $\epsilon>0$ such that if there exists $j$ with $x_j<\epsilon$ then \smash{$\rho-\sum_{i=1}^{k-1} x_i>1$}. Making $\epsilon>0$ smaller, if necessary, we may assume that $\delta=1-\rho+(1-\epsilon)(k-1)$ lies in $(0,1)$. Observing that 
$\nu$ is finite on $[\epsilon,1]$ and $h$ bounded on compact subintervals of $(0,1)$ 
it only remains to show that
$$\int_{[\epsilon,1]} \nu(dx_1) \ldots  \int_{[\epsilon,1]} \nu(dx_{k-1}) 1_{\rho-\sum_{i=1}^{k-1} x_i> 1 -\delta } 
h(\rho-\sum_{i=1}^{k-1} x_i) <\infty.$$
If $x_i>1-\epsilon$ for all $i$, then $\rho-\sum_{i=1}^{k-1} x_i <1 -\delta$. Hence,
\begin{align*}
    \int_{[\epsilon,1]}& \nu(dx_1) \cdots  \int_{[\epsilon,1]} \nu(dx_{k-1}) 1_{\rho-\sum_{i=1}^{k-1} x_i> 1 -\delta } 
h(\rho-\sum_{i=1}^{k-1} x_i) \\
& \leq 
(k-1)  \int_\epsilon^{1-\epsilon} dx_1 \int_{[\epsilon,1]} \nu(dx_{2})\cdots  \int_{[\epsilon,1]} \nu(dx_{k-1}) 1_{\rho-\sum_{i=1}^{k-1} x_i> 1 -\delta } 
h(x_1)
h(\rho-\sum_{i=1}^{k-1} x_i)\\
& \leq (k-1) \Big(\max_{[\epsilon, 1-\epsilon]}h \Big) (\nu[\epsilon,1])^{k-2}
\int^1_{1-\delta}  h(x) \, dx,
\end{align*}
which is finite because $h$ is the density of the finite measure $\nu|_{(1-\delta,1)}$ by assumption.\smallskip

Equation~\eqref{eq:lln_wn} is a well-known result and can be derived from 
a law of large numbers for triangular  schemes, for example \cite[Theorem 3, Chapter IX]{Petrov1975}. The weak law of large numbers states that the random variables $(\frac{1}{n}(W_k^{_{(n)}} - \mu_n) \colon 1\leq k\leq n)$ satisfy $\sum_{k=1}^n \frac{1}{n}(W_k^{_{(n)}} - \mu_n) \to 0$ in probability, if, for all $\eps>0$ and $\tau >0$, we have
\begin{enumerate}[label=(\roman*)]
    \item $ n \mathbb P(W^{(n)} > \eps n + \mu_n) \to 0, $ \label{eq:lln_wn_cond1}\smallskip
    \item $\text{Var}(\frac{1}{n}(W^{(n)} - \mu_n)\I{W^{(n)} < \tau n + \mu_n}) \to 0,\, $\label{eq:lln_wn_cond2}\smallskip
    \item $\mathbb E\big[ \frac{1}{n}(W^{(n)} -\mu_n)\I{W^{(n)} < \tau n + \mu_n}\big] \to 0.$\label{eq:lln_wn_cond3}
\end{enumerate}
To verify these conditions, recall from~\eqref{eq:lp_bound_assumption} that $\prob(W^{(n)}> \eps n +\mu_n)  =  ( 1  +o(1)) \, f(n)\, \nu(\eps,1]$. As $(f(n))_n$ is regularly varying with index $-\alpha<-1$, this implies~\ref{eq:lln_wn_cond1}. To see why \ref{eq:lln_wn_cond2} is satisfied, abbreviate 
\smash{$U^{(n)} = (W^{(n)} - \mu_n){\mathbf 1}[W^{(n)} < \tau n + \mu_n]$}. Since $|U^{(n)}| \leq \tau n + \mu_n$ we get $\mathbb E[|U^{(n)}|^2] \leq \mathbb E[|U^{(n)}|](\tau n + \mu_n)  \leq 2\mu_n (\tau n + \mu_n)$, which implies~\ref{eq:lln_wn_cond2}. Lastly, we have 
\begin{align*}
    \tfrac{1}{n}\mathbb E\big[ U^{(n)} \big] & = \tfrac{1}{n}\mu_n \mathbb P(W^{(n)} \geq \tau n + \mu_n) -  \tfrac{1}{n}\mathbb E\big[W^{(n)}\I{W^{(n)} \geq \tau n + \mu_n}\big].
\end{align*}
The first term on the right-hand side tends to zero since $\mu_n = o(n)$. For the second term, observe that $\mathbb E[W^{(n)}{\mathbf 1}[W^n > \tau n + \mu_n]] \leq n \mathbb P(W^{(n)} \geq \tau n + \mu_n)$ tends to zero by \ref{eq:lln_wn_cond1}, thus condition \ref{eq:lln_wn_cond3} holds.
\smallskip

We now give a local limit theorem for \emph{finite} sums of truncated random variables. Throughout the section, we use the abbrevation $[n]$ to denote the set $\{1,\ldots,n\}$.
\begin{lemma}\label{lem:sk_approximation}
Let $k$ be a fixed integer and suppose $W_1^{_{(n)}}, \ldots, W_k^{_{(n)}}$ are independent, identically distributed and 
satisfy Assumption~\ref{eq:distr_assumption}.
Let $$T_k = T_k(n) := W_1^{_{(n)}}+\cdots +W_k^{_{(n)}}.$$
    For any  $k-1<\rho<k$ and sequences $\sigma_1(n) \uparrow \rho, \sigma_2(n) \downarrow \rho$ with {$\sigma_2(n)-\sigma_1(n) \gg \eta_n$},  we have that
     \begin{align*}
        \p{n\sigma_1(n) \leq T_k \leq n\sigma_2(n)} & = (K_{\rho}+o(1))(\sigma_2(n) - \sigma_1(n)) \, {f(n)}^{k}.
    \end{align*}
\end{lemma}
\begin{proof}[Proof of Lemma~\ref{lem:sk_approximation}]
   Let $\sigma_1(n) \uparrow \rho, \sigma_2(n) \downarrow \rho$ be such that $\sigma_2(n) - \sigma_1(n) \gg \eta_n$. 
   Fix a small $\eps>0$ and let $\Delta_n \downarrow 0$ be a sequence satisfying $\eta_n \ll \Delta_n \ll \sigma_2(n) -\sigma_1(n)$ and $m = (1-\eps)/\Delta_n \in \N$. Fix a partition $P_n = \{t_0,\ldots ,t_m\}$ of $[\eps,1]$ satisfying $|t_{i}-t_{i-1}| = \Delta_n$ for all $i\in  [m]$. We begin by constructing a lower bound of $\p{n\sigma_1(n)\leq T_k \leq n\sigma_2(n)}$.  Note that we have
     \begin{align*}
         & \p{n\sigma_1(n)\leq T_k \leq n\sigma_2(n)}\\
         & \geq \sum_{\substack{(\ell_1,\ldots,\ell_k) \in [m]^k}}\prod_{j=1}^k\p{nt_{\ell_j-1} < W_j\hn \leq nt_{\ell_j}}\I{ \sigma_1(n) \leq \sum_{j\in [k]}t_{\ell_j-1}, \sum_{j\in [k]}t_{\ell_j}\leq \sigma_2(n)  }\\
         & = (1+o(1)) f(n)^{k}\sum_{\substack{(\ell_1,\ldots,\ell_k)\in [m]^k\\ \sum_{j\in [k]}t_{\ell_j-1}\geq \sigma_1(n) \\ \sum_{j\in [k]}t_{\ell_j}\leq \sigma_2(n) }}
         \prod_{j=1}^k \nu(
         t_{\ell_j-1}, t_{\ell_j}] \\
         & \geq (1+o(1))  f(n)^{k}\nu^k \big\{(x_1,\ldots,x_k)\in(\eps,1]^k \colon \sigma_1(n) + k\Delta_n \leq \sum_{j=1}^k x_j \leq  \sigma_2(n) - k\Delta_n \big\},
    \end{align*}
    where the equality holds by  Assumption~\ref{eq:distr_assumption}. To see why the last inequality holds, note that as $P_n$ is a partition of $[\eps,1]$, for any $(x_1,\ldots,x_k)\in (\eps,1]^k$ we can find $(\ell_1,\ldots,\ell_k) \in [m]^k$ such that $x_i \in (t_{\ell_i-1},t_{\ell_i}]$ for each $i\in [k]$. Further, we know that $t_{\ell_i-1} = t_{\ell_i } - \Delta_n$. Thus if $x_1+\cdots +x_k \geq  \sigma_1(n) + k\Delta_n$ then $ \sum_{j\in [k]}t_{\ell_j-1} =  \sum_{j\in [k]}t_{\ell_j} - k\Delta_n \geq  \sum_{j\in [k]}x_j - k\Delta_n \geq \sigma_1(n)$. Similarly if $x_1+\cdots+x_k \leq  \sigma_2(n) - k\Delta_n$ then $\sum_{j\in [k]}t_{\ell_j}\leq \sigma_2(n)$. 

Given $\bx=(x_1,\ldots,x_k)\in(\eps,1]^k$ we let $j(\bx)$
be the number of entries in $\bx$ that are equal to one and let $c:=\nu(\{1\})$. Note that for $j(\bx) = k$, the condition $ \sum_{j=1}^k x_j \leq  \sigma_2(n) - k\Delta_n $ is not satisfied for large $n$. For $j \in \{0,\ldots, k-1\}$, we can write
    \begin{align*}
        & \nu^k 
        \big\{\bx = (x_1,\ldots,x_k)\in (\eps,1]^k \colon j(\bx)=j,\, \sigma_1(n) + k\Delta_n \leq 
        \textstyle \sum_{i=1}^k x_i \leq  \sigma_2(n) - k\Delta_n \big\}\\
         & = c^j {k \choose j} 
         \int_{(\eps,1)} \nu(dx_1)\cdots 
         \int_{(\eps,1)} \nu(dx_{k-j-1}) \int_{ \sigma_1(n) + k\Delta_n - (x_1+\cdots +x_{k-j-1}+j)}^{ \sigma_2(n) - k\Delta_n -(x_1+\cdots +x_{k-j-1}+j)}h\Inb{(\eps,1)}(x)  \, dx. 
      \end{align*}    
    Since 
    $\sigma_2(n) - k\Delta_n$ and $\sigma_1(n) + k\Delta_n$ both converge to $\rho>k-1$ as $n$ tends to infinity,  we have $\sigma_2(n) - \sigma_1(n) - 2k\Delta_n = (1+o(1))(\sigma_2(n) - \sigma_1(n))$ by choice of $\Delta_n$. This is asymptotically 
    equivalent to 
    $$(\sigma_2(n) - \sigma_1(n) )  {k \choose j} c^j\int_{(\eps,1)} \nu(dx_1)\cdots 
         \int_{(\eps,1)} \nu(dx_{k-j-1}) 
\, h\Inb{(\eps,1)}\Big(\rho-j-\sum_{i=1}^{k-j-1}x_i\Big).$$
    By the monotone convergence  theorem, since $h$ is non-negative, the integral converges, as $\eps\downarrow0$, to
    $$\int_{(0,1)} \nu(dx_1)\cdots 
         \int_{(0,1)} \nu(dx_{k-j-1}) 
\, h\Big(\rho-j-\sum_{i=1}^{k-j-1}x_i\Big).$$
Splitting each integration with respect $\nu$ into
an integral over $(0,1)$ plus a contribution of the mass at one we get, for any permutation invariant function $f\colon (0,1]^{k-1}\to (0,\infty)$, that
\begin{align*}
\int \nu(dx_1)\cdots & \int \nu(dx_{k-1}) 
f(x_1,\ldots, x_{k-1}) \\
& = \sum_{j=0}^{k-1} \binom{k-1}{j} c^j 
\int_{(0,1)} \nu(dx_1)\cdots \int_{(0,1)} 
\nu(dx_{k-1-j}) f(x_1,\ldots, x_{k-1-j}, \underbrace{1, \ldots, 1}_{j \text{ times}}).
\end{align*}
\vspace{-6mm}

\noindent
Applying this to $f(x_1,\ldots, x_{k-1})=\frac{k}{k-\sum_{i=1}^{k-1}\I{x_i=1}}
h\big(\rho-\sum_{i=1}^{k-1}x_i\big)$ gives
\begin{align*}
\sum_{j=0}^{k-1}   \binom{k}{j} c^j \int_{(0,1)} \nu(dx_1)\cdots &
         \int_{(0,1)} \nu(dx_{k-j-1}) 
\, h\Big(\rho-j-\sum_{i=1}^{k-j-1}x_i\Big) \\
& =  \int \nu(dx_1)\cdots 
         \int \nu(dx_{k-1}) \, \frac{k}{k-\sum_{i=1}^{k-1}\I{x_i=1}} \, 
\, h\Big(\rho-\sum_{i=1}^{k-1}x_i\Big),
\end{align*}
    and the desired lower bound follows.\smallskip

    Next, we construct an upper bound. Let $0<\eps < \rho - (k-1)$ and  observe that for $n$ sufficiently large $n\sigma_1(n)\leq T_k \leq n\sigma_2(n)$ implies
    $W_i\hn \ge \eps n$ for all $i \in [k]$.
    It follows that, for large $n$,
    \begin{align*}
        & \p{n\sigma_1(n)\leq T_k \leq n\sigma_2(n)} \\
        &  = \p{n\sigma_1(n)\leq T_k \leq n\sigma_2(n)  \text{ and }  W_i\hn \geq \eps n\,
        \, \text{for all } i \in [k]} \\
        & \leq \sum_{\substack{(\ell_1,\ldots,\ell_k)\in [m]^k }}\prod_{j=1}^k\p{nt_{\ell_j-1} < W_j\hn \leq nt_{\ell_j}}\I{ \sum_{j\in [k]}t_{\ell_j-1}\leq \sigma_2(n), \sum_{j\in [k]}t_{\ell_j}\geq \sigma_1(n)}\\
        & =  (1+o(1))  f(n)^{k} \sum_{\substack{(\ell_1,\ldots,\ell_k)\in [m]^k\\ \sum_{j\in [k]}t_{\ell_j-1}\leq \sigma_2(n) \\ \sum_{j\in [k]}t_{\ell_j}\geq \sigma_1(n) }}  \prod_{j=1}^k \nu(
         t_{\ell_j-1}, t_{\ell_j}] \\
        & \leq (1+o(1)) f(n)^{k}
         \nu^k \big\{(x_1,\ldots,x_k)\in(\eps,1]^k \colon \sigma_1(n) - k\Delta_n \leq \sum_{j=1}^k x_j \leq  \sigma_2(n) + k\Delta_n \big\}.
    \end{align*}
    The last inequality follows from the fact that $\sum_{j\in[k]}t_{\ell_j} \geq \sigma_1(n)$ implies $ \sigma_1(n) \leq \sum_{j\in[k]}t_{\ell_j} = \sum_{j\in[k]}t_{\ell_j-1} + k\Delta_n \leq \sum_{j\in[k]}x_j + k\Delta_n$. Similarly, if we have $\sum_{j\in [k]}t_{\ell_j-1}\leq \sigma_2(n)$, then $\sum_{j\in[k]}x_j$ $\leq  \sigma_2(n) + k\Delta_n$. 
We bound the measure from above following the same reasoning as in the lower bound, in this case using that $\sigma_2(n) - \sigma_1(n) + 2k\Delta_n = (1+o(1))(\sigma_2(n) - \sigma_1(n))$
and bounding integrals
over $(\eps,1)$ by $(0,1)$, to obtain a matching upper bound.
\qedhere
\end{proof}

With these tools at hand we now focus on the proof of Theorem~\ref{thm:main_theorem}. 
Let $\delta_n \gg \zeta_n \vee \eta_n$ with 
$\zeta_n \downarrow 0$ be a sequence given by~\eqref{eq:lln_wn}. As previously remarked, we can assume that $\delta_n > n^{1-\alpha} \vee n^{-1}$ for all $n\geq 0$. Let  $\rho_1(n)\uparrow\rho$ and $\rho_2(n)\downarrow \rho$ be sequences satisfying $\rho_2(n)-\rho_1(n)\gg \delta_n$. Fix $\eps$ such that 
\begin{align}\label{eq:epsilon_restriction}
    0< \eps < (\alpha -1)(\rho - (k-1))/(2 + k\alpha)
\end{align}
and recall the definition of the interval $I_n = I_n(\rho_1(n),\rho_2(n))$ given by \eqref{def:interval_in}. The choice of $\eps$ will be relevant in the proof of Claim \ref{claim:bound_e4j}. Further, for $m\leq n$ define $T_m = T_m(n) := W_1\hn + \cdots + W_m\hn$. To accommodate for error terms we introduce two sequences $(\rho^{\subset}_1(n))_n$ and  $(\rho^{\subset}_2(n))_n$ converging to $\rho$ satisfying $\rho_1(n) + \delta_n < \rho^{\subset}_1(n) < \rho_1(n) + 2\delta_n$ and $ \rho_2(n) -2 \delta_n < \rho^{\subset}_2(n) < \rho_2(n) - \delta_n$ for all $n\geq 1$. Note that $\rho^{\subset}_2(n) - \rho^{\subset}_1(n) = (1+o(1))(\rho_2(n) - \rho_1(n))$. We decompose $\{ S_n \in I_n\}$ into disjoint events and study their individual probabilities. Define the following events.
\begin{align*}
    \cE_1 :\, & S_n \in I_n, \text{ there exist exactly $k$ indices $i_1,\ldots ,i_k \in [n]$ such that} \\ &  W_{i_1}\hn> \eps  n,\ldots, W_{i_k}\hn > \eps  n, \text{
    and } W_{i_1}\hn + \cdots + W_{i_k}\hn  \in [n\rho^{\subset}_1(n), n\rho^{\subset}_2(n)].\\
    \cE_2 :\,& S_n \in I_n, \text{ there exist exactly $k$ indices $i_1,\ldots,i_k \in [n]$ such that }\\
    & \text{$W_{i_1}\hn> \eps  n$,\ldots, $W_{i_k}\hn > \eps  n$,
    and }  W_{i_1}\hn + \cdots + W_{i_k}\hn  \not\in [n\rho^{\subset}_1(n), n\rho^{\subset}_2(n)].\\
    \cE_3 :\,& S_n \in I_n, \text{ $W_i\hn > \eps  n$ for at least $k+1$ indices $i \in [n]$.}\\
    \cE_4^{j} :\,& S_n \in I_n, \text{  there are exactly $j$ indices $i_1,\ldots,i_j \in [n]$ such that }\\
    & W_{i_1}\hn> \eps  n,\ldots,W_{i_j}\hn> \eps  n.
\end{align*}
We can then rewrite $\{S_n \in I_n\}$ as a union of disjoint events,
\begin{align}\label{eq:sn_disjoint_events}
    \{S_n \in I_n\} = \cE_1 \cup \cE_2 \cup \cE_3 \cup \bigcup_{j=0}^{k-1} \cE_4^j.
\end{align}
The proof consists in showing that $\cE_1$ is the dominating event and that the probabilities of the events $\cE_2, \cE_3, \cE_4^j$ for $0\leq j \leq k-1$ are of smaller order than $\p{\cE_1}$.  Define the following events for $\{i_1,\ldots,i_{k+1}\}\subset [n]$,
\begin{align*}
    \cE_{1;i_1\ldots i_k}:\, &  S_n   \in I_n,\,  W_{i_1}\hn + \cdots + W_{i_k}\hn  \in [n\rho^{\subset}_1(n), n\rho^{\subset}_2(n)]\ \text{ and }\\
    & W_{i_1}\hn> \eps  n, \ldots,  W_{i_k}\hn> \eps  n,\, W_j\hn \leq \eps  n \,  \forall  j\notin \{i_1,\ldots,i_k\} ,\\
    \cE_{2;i_1\ldots i_k}:\, & S_n   \in I_n,\, W_{i_1}\hn + \cdots + W_{i_k}\hn  \notin [n\rho^{\subset}_1(n), n\rho^{\subset}_2(n)]\text{ and }\\  &  W_{i_1}\hn> \eps  n, \ldots,  W_{i_k}\hn> \eps  n,\\
    \cE_{3;i_1 \ldots i_ki_{k+1}}:\, & S_n   \in I_n, W_{i_1}\hn>  \eps  n,\ldots, W_{i_{k+1}}\hn>  \eps  n.
\end{align*}
The identities below follow immediately.
\begin{align}%
    \p{\cE_1} & = {n\choose{k}}\p{\cE_{1;12\ldots k}}, \label{eq:e1_equality_e11k}\\
    \label{eq:e2_inequality_d1k}
    \p{\cE_2} & \leq {n\choose{k}}\p{\cE_{2;12\ldots k}},\\
     \label{eq:e3_inequality_estar21k}
    \p{\cE_3} & \leq {n\choose{k+1}}\p{\cE_{3;12\ldots k\,k+1}}.
\end{align}
Further, we claim that the following equalities hold. 
\begin{claim}\label{claim:new_bound_e121k}
    $\p{\cE_{1;12\ldots k}} = (1+o(1))\mathbb P\big(n\rho^{\subset}_1(n) \leq T_k \leq  n\rho^{\subset}_2(n)\big).$
\end{claim}
\begin{claim}\label{claim:bound_estar21k}
    $\p{\cE_{2;12\ldots k}}  = o((\rho_2(n) - \rho_1(n)) f(n)^{k}).$
\end{claim}
\begin{claim}\label{claim:bound_d12k} 
    $\p{\cE_{3;12\ldots k\,k+1}}  = o((\rho_2(n) - \rho_1(n)) n^{- 1} f(n)^{k}).$
\end{claim}
\begin{claim}\label{claim:bound_e4j}
    For $0\leq j < k$, we have $\mathbb P(\cE^j_{4}) = o(n^{-1}n^{k}f(n)^{k}).$
\end{claim}
By (\ref{eq:e1_equality_e11k}), Claim~\ref{claim:new_bound_e121k} and Lemma~\ref{lem:sk_approximation} we then have that 
\[\p{\cE_1} =(1+o(1)){n\choose{k}} (\rho^{\subset}_2(n) - \rho^{\subset}_1(n)) f(n)^k K_{\rho}.\]
Combining the identities and claims above, together with (\ref{eq:sn_disjoint_events}), the fact that $\rho^{\subset}_2(n) - \rho^{\subset}_1(n) = (1+o(1))(\rho_2(n) - \rho_1(n))$ and $\rho_2(n) - \rho_1(n) \gg \delta_n > n^{-1}$ gives 
\[\p{S_n \in I_n} = (1+o(1))\p{\cE_1} = ( K_{\rho}+o(1)) {n \choose k} (\rho_2(n)-\rho_1(n)) f(n)^{k}, \] 
proving Theorem~\ref{thm:main_theorem}. Corollary~\ref{cor:dominating_event} follows by choosing 
$\gamma_n = \max\{\rho^{\subset}_2(n) - \rho_1(n), \rho_2(n) - \rho^{\subset}_1(n)\}$.
The remainder of this section focuses on proving the Claims \ref{claim:new_bound_e121k} to \ref{claim:bound_e4j}.
   
\begin{proof}[Proof of Claim~\ref{claim:new_bound_e121k}] The claimed upper bound on $\p{\cE_{1;12\ldots k}}$ follows directly from the definition of the event. To prove a lower bound, remark that 
\begin{align*}
     \{n(\rho_1 & + \mu)  \leq S_n \leq n(\rho_2 + \mu),\, n\rho^{\subset}_1(n)\leq T_k \leq n \rho^{\subset}_2(n) \}\\
     & \supseteq  \{n(\mu - (\rho^{\subset}_1(n) - \rho_1(n)) ) \leq \sum_{i=k+1}^nW_i\hn \leq n( \mu + (\rho_2(n) -\rho^{\subset}_2(n)))\}\\
     & \phantom{gobbldeegobldeegobldeebibbldibumdigok} \cap \{n\rho^{\subset}_1(n)\leq T_k \leq n \rho^{\subset}_2(n) \}\\
     & \supseteq \{|\sum_{i=k+1}^nW_i\hn - n\mu | \leq n\delta_n\}\cap  \{n\rho^{\subset}_1(n)\leq T_k \leq n \rho^{\subset}_2(n) \},
\end{align*}
where the last step holds since $\min \{ \rho^{\subset}_1(n) - \rho_1(n), \rho_2(n) -\rho^{\subset}_2(n)\} \geq \delta_n$. The preceding inclusion together with the fact that the $\smash{W^{(n)}_1,\ldots, W^{(n)}_n}$ are i.i.d.\,implies that 
\begin{align*}
    \p{\cE_{1;12\ldots k}} \geq & \p{|T_{n-k} -n\mu| \leq n\delta_n,\,   W_i\hn \leq \eps n\, \forall i \in [n-k]}\\
    & \cdot \p{n\rho^{\subset}_1(n)\leq T_k \leq n \rho^{\subset}_2(n),\,  W_i\hn > \eps n \,\forall i \in [k]}.
\end{align*}
The second probability of the right hand side equals that of $\{n\rho^{\subset}_1(n)\leq T_k \leq n \rho^{\subset}_2(n)\}$ for sufficiently large $n$ as, if $W_i\hn \leq \eps n$ for some $i\in [k]$, then for large $n$, $T_k \leq n(k-1+\eps)< n\rho^{\subset}_1(n)$. Further, 
\begin{align*}
     \prob(|T_{n-k} -n\mu| & \leq n\delta_n,\,   W_i\hn \leq \eps n\, \forall i\in [n-k]) \\
    & \geq \p{|T_{n-k} - n\mu|\leq n\delta_n} -  (n-k)\mathbb P(W\hnn > \eps n).
\end{align*}
 By \eqref{eq:lln_wn}, $\p{|T_{n-k} - n\mu|\leq n\delta_n} \geq 1 -\delta_n$. 
From \eqref{eq:lp_bound_assumption} we obtain that $(n-k)\mathbb P(W\hn > \eps n)$ converges to zero. The lower bound on $\mathbb P(\cE_{1;12 \ldots k})$ then directly follows from the above.
\end{proof}

\begin{proof}[{Proof of Claim~\ref{claim:bound_estar21k}}]
    Let $\smash{(\rho^{\supset}_1(n))_n}$ and $\smash{(\rho^{\supset}_2(n))_n}$ be two sequences converging to $\rho$ satisfying $ \rho_1(n) - 2\delta_n < \rho^{\supset}_1 < \rho_1(n) - \delta_n$ and $\rho_2(n) + \delta_n < \rho^{\supset}_2 < \rho_2(n) + 2\delta_n$. By construction $|\rho^{\subset}_1(n) - \rho^{\supset}_1(n)|, |\rho^{\supset}_2(n) - \rho^{\subset}_2(n)| \in [2\delta_n,4\delta_n]$ for all $n\geq 1$.  
    Lemma~\ref{lem:sk_approximation} gives that 
    \begin{align*}
        \p{n\rho^{\supset}_1(n) \leq T_k \leq n\rho^{\subset}_1(n)},\p{n\rho^{\subset}_2(n) \leq T_k \leq n\rho^{\supset}_2(n)} =
        O(\delta_n n^{-\alpha k}),
    \end{align*}
    which is 
    $o(( \rho_2(n) - \rho_1(n))n^{-\alpha k})$ and therefore negligible. Thus we can focus on bounding the term 
    $ \mathbb P(\cE_{2;1\ldots k},\, T_k \notin [n\rho^{\supset}_1(n), n\rho^{\supset}_2(n)])$. Let $\eps>0$ and let $\Delta_n \downarrow 0$ be a sequence satisfying $\eta_n \ll \Delta_n \ll \delta_n$ such that $m = (1-\eps)/\Delta_n
    \in\N$. Fix a partition $P_n = \{t_0,\ldots ,t_m\}$ of $[\eps,1]$ satisfying $|t_{i}-t_{i-1}| = \Delta_n$ for all $ i\in [m]$. Observe that 
     \begin{align*}
        & \p{\cE_{2;12\ldots k}, \, T_k \notin [n\rho^{\supset}_1(n), n\rho^{\supset}_2(n)]} \\
        & \leq \sum_{\substack{(\ell_1,\ldots ,\ell_k)\\ \sum_{j\in[k]}t_{\ell_j-1} \leq \rho^{\supset}_1(n)\\
        \text{ or }\sum_{j\in[k]}t_{\ell_j} \geq \rho^{\supset}_2(n)}}\p{S_n \in I_n,\, W_i\hn \in (nt_{\ell_i-1}, nt_{\ell_i}]\, \forall i\in [k]}.
    \end{align*}
    Note that the event $\{S_n \in I_n,\, W_i\hn \in (nt_{\ell_i-1}, nt_{\ell_i}]\,  \forall i \in [k]\}$ is contained in the event $ \{\sum_{i=k+1}^nW\hn_i \in [n(\rho_1 + \mu - \sum_{j\in[k]}t_{\ell_j}), n(\rho_2 + \mu - \sum_{j\in[k]}t_{\ell_j -1})]\}\cap \{W_i\hn \in (nt_{\ell_i -1}, nt_{\ell_i}] \, \forall i\in [k]\}$.  Since $t_{j} \in [\eps , 1]$ for all $j \in \{0,\ldots,m\}$, we can approximate $\mathbb P(W_i\hn \in (nt_{j-1}, nt_{j}])$ using Assumption~\ref{eq:distr_assumption}. As $(W_i\hn \colon i\in [k])$ are i.i.d.\,, we have
    \begin{align*}
         \prob\big(  S_n \in I_n,\, & W_i\hn \in (nt_{\ell_i-1}, nt_{\ell_i}]\,  \forall i\in [k]\big) \\
         & \leq (1+o^*(1)) f(n)^{k}
        \prod_{i=1}^k \nu(t_{\ell_i-1},t_{\ell_i}]  \\  
        & \phantom{XXXX}
        \P\Big(T_{n-k} \in [n(\rho_1 + \mu - \sum_{j\in [k]}t_{\ell_j}), n(\rho_2 + \mu - \sum_{j\in [k]}t_{\ell_j-1})]\Big).
    \end{align*}
    Suppose $\sum_{j\in [k]}t_{\ell_j-1} \leq \rho^{\supset}_1(n)$. As $ \rho_1(n) -\rho^{\supset}_1(n) > \delta_n$, the latter probability can be bounded  above by $\p{T_{n-k} > n(\mu + \delta_n - k\Delta_n)} = \p{T_{n-k} - n\mu >(1+o(1))n\delta_n}$. Similarly if $\sum_{j\in [k]}t_{\ell_j} \geq \rho^{\supset}_2(n)$, we get $\p{T_{n-k} - n\mu < - (1+o(1))n\delta_n}$ as an upper bound. As $\nu(\eps,1]$ is finite 
    we have
    \begin{align*}
    \sum_{\substack{(\ell_1,\ldots ,\ell_k)\\ \sum_{j\in[k]}t_{\ell_j-1} \leq \rho^{\supset}_1(n)\\
        \text{ or }\sum_{j\in[k]}t_{\ell_j} \geq \rho^{\supset}_2(n)}}\prod_{i=1}^k \nu(t_{\ell_i-1},t_{\ell_i}] \leq (\nu(\eps,1])^k.
    \end{align*}
By~\eqref{eq:lln_wn} and using the fact that $\delta_n \ll \rho_2(n) - \rho_1(n)$, we infer that
    \begin{align*}
        \mathbb P\big(\cE_{2;12\ldots k},\, T_k \notin &[n\rho^{\supset}_1(n), n\rho^{\supset}_2(n)]\big) \\
        & \leq  (1+o(1))(\nu(\eps,1])^k f(n)^k \,\p{|T_{n-k} -n\mu| > (1+o(1))n\delta_n}\\
        & = o((\rho_2(n) - \rho_1(n))f(n)^k).\qedhere
    \end{align*}
\end{proof}

\begin{proof}[Proof of Claim~\ref{claim:bound_d12k}]
      Let $\eps>0$ and let $\Delta_n \downarrow 0$ be a sequence satisfying $\Delta_n \gg \eta_n$ such that $m = (1-\eps)/\Delta_n \in \N$. Fix a partition $P_n = \{t_0,....,t_m\}$ of $[\eps,1]$ satisfying $|t_{i}-t_{i-1}| = \Delta_n$ for all $i\in[m]$.  Then by Assumption~\ref{eq:distr_assumption},
     \begin{align*}
         \p{\cE_{3;12\ldots k\,k+1}} & = \sum_{(\ell_1,\ldots ,\ell_{k+1})}\p{S_n \in I_n,\, W_i\hn \in (nt_{\ell_i-1}, nt_{\ell_i}] \,\forall i\in [k+1]}\\
         & \leq  \sum_{(\ell_1,\ldots ,\ell_{k+1})}\p{W_i\hn \in (nt_{\ell_i-1}, nt_{\ell_i}]\,\forall i\in [k+1]}\\
         & = (1+o(1))f(n)^{k+1} \sum_{(\ell_1,...,\ell_{k+1})}\prod_{i=1}^{k+1} \nu(t_{\ell_i-1},t_{\ell_i}]  \\
         & \leq  (1+o(1))(\nu(\eps,1])^{k+1} f(n)^{k+1}.
    \end{align*}
    The claim follows by choosing $\delta_n \gg nf(n)$ and noting that $f(n) \ll (\rho_2(n) - \rho_1(n))n^{-1}$.
\end{proof}

\begin{proof}[Proof of Claim~\ref{claim:bound_e4j}]
    Fix $0\leq j\leq k-1$ and note that 
    \[\p{\cE_4^j} = {n\choose{j}}\p{S_n \in I_n, \, W_i\hn \leq \eps  n\, \forall i\in [n-j],\, W_i\hn > \eps  n \, \forall n-j<i\leq n}.\]
    Since $W\hnn \leq n$, we have that $T_j =W_1^{_{(n)}}+\cdots +W_j^{_{(n)}}\leq jn$ and so
    \begin{align*}
        \p{\cE_4^j} & \leq {n\choose{j}} \p{ W\hnn > \eps  n}^j \p{T_{n-j} \geq n(\rho_1(n) - j + \mu), \,  W_i\hn \leq \eps  n \,\forall i \in [n-j]}.
    \end{align*}
By \eqref{eq:lp_bound_assumption} for  any $1<  p < \alpha$ we can find a constant $C>0$ such that 
    $\mathbb P( W\hnn > \eps  n) \leq  C n^{-p}.$
Thus it suffices to prove that for all $0\leq j\leq k-1$,
    \begin{align}\label{eq:upperbound_order_e4j_tn}
        \p{T_{n-j} \geq n(\rho_1(n) - j + \mu), \,  W_i\hn \leq \eps  n \, \forall i \in [n-j]} = o( n^{-1}n^{k + j(p-1)} f(n)^k).
    \end{align}
   Define the truncated random variable $$\hat{W}\hnn := W\hnn {\bf 1}_{[W\hnn \leq \eps  n]} \text{ and } \hat{T}_k := \sum_{i =1}^k\hat{W}_i\hn,$$ where $(\hat{W}_i\hn \colon i\geq 1)$ are i.i.d. random variables, distributed like $\hat{W}\hnn$. By construction, for any $m>0$ we have
     $\{T_{n-j} \geq m, W_i\hn \leq \eps  n \,\forall i\in [n-j]\}\subseteq \{\hat{T}_{n-j} \geq m\}.$
    By Chebyshev's inequality, for all $s>0$,
     \begin{align*}
        \p{\hat{T}_{n-j} \geq m} & = \p{\sum_{i =1}^{n-j} \hat{W}_i\hn \geq m} \leq e^{-sm}\E{\exp(s\hat{W}^{(n)})}^{n-j}.
    \end{align*}
    Let $s_n := a \log n/n$, for a constant $a>0$ chosen later. Using the Taylor series expansion of $e^x$, we get
    \begin{align*}
      \E{\exp(s_n\hat{W}^{(n)})} 
      \leq 1 + s_n\mu_n + \mathbb E\Big[\sum_{k\geq 2}\frac{(s_n\hat{W}^{(n)})^k}{k!}\Big].
    \end{align*}
    Suppose for now that for $a < (\alpha - 1)/\eps $, we have that as $n\rightarrow \infty$,
    \begin{align}\label{eq:exp_sum_taylor_order_s}
         \mathbb E\Big[\sum_{k\geq 2}\frac{(s_n\hat{W}^{(n)})^k}{k!}\Big] & = o(s_n).
    \end{align}
    From this it follows that 
$\mathbb E\big[\exp(s_n\hat{W}^{(n)}) \big]  \leq  1 + s_n\mu_n + o(s_n) \leq \exp{(s_n(\mu_n + o(1)))}$.
    Putting everything together,
    \begin{align*}
         & \p{T_{n-j} \geq n(\rho_1(n) - j + \mu), \,  W_i\hn \leq \eps  n \,\forall i \in [n-j]} \\
         & \leq  \prob\big( \hat{T}_{n-j} \geq n(\rho_1(n) - j + \mu) \big)\\
         & \leq \exp(- s_nn(\rho_1(n) - j + \mu) + s_n(n-j)(\mu_n + o(1)) ) = \exp(s_nn(- (\rho - j + o(1))),
    \end{align*}
    where we use the fact that $\mu = \mu_n + o(1)$ and $\rho = \rho_1(n) + o(1)$. As $s_n = a\log n / n$, we have that $ \exp(s_nn(- (\rho - j + o(1)))  = o(n^{1-a(\rho - j))})$. Note that since $f$ is a regularly varying function with index $-\alpha$, for all $b>\alpha$ we have that $-\log_n(f(n)) < b$ for sufficiently large~$n$. Let $a = (2+k\alpha)/(\rho-(k-1))$. Note that $a<(\alpha -1 )/\eps$ by our choice of $\eps$ given by (\ref{eq:epsilon_restriction}). The equality (\ref{eq:upperbound_order_e4j_tn}) then follows for all $0\leq j \leq k-1$ since
    \begin{align*}
        1 - a(\rho-j) & \leq -1 + k + j(p-1) + k\log_n(f(n)),
    \end{align*}
    using the fact that $-\log_n(f(n)) < \alpha + 1$ for sufficiently large $n$. It remains to prove (\ref{eq:exp_sum_taylor_order_s}). We can choose constants $c, C>0$ such that $e^{x} \leq 1 + x + Cx^2$ for all $x\in (0, c)$. Then, for sufficiently large $n$, 
    \begin{equation}\label{eq:sum_exp}
         \mathbb E\Big[\sum_{k\geq 2}\frac{(s_n\hat{W}^{(n)})^k}{k!}\Big]  \leq Cs_n^2\mathbb E\Big[(W\hnn )^2\I{W^{(n)}\leq \frac{c}{s_n}}\Big] + \mathbb E\Big[e^{s_nW\hnn }\I{W\hnn \in (\frac{c}{s_n}, \eps  n]}\Big].
    \end{equation}
    As $(W\hnn)_n$ is $L^{2-\beta}$-bounded for $2-\alpha < \beta < 1$, we have
     \begin{align*}
        s_n^2\E{(W\hnn )^2\I{W^{(n)}\leq c/s_n}} & = s_n^2\E{(W\hnn )^{\beta}(W\hnn )^{2-\beta}\I{W\hnn \leq c/s_n}} \\
        &\leq c^{\beta}s_n^{2-\beta}\E{(W\hnn )^{2-\beta}}  = o(s_n),
    \end{align*}
    where the last step holds since $s_n$ converges to $0$ as $n$ increases, thus $s_n^{2-\beta} = o(s_n)$. The second summand of (\ref{eq:sum_exp}) can be bounded, for $1< \alpha' < \alpha$, by
    \begin{align*}
        \E{e^{s_nW\hnn }\I{W\hnn \in (\frac{c}{s_n}, \eps  n]}}   & \leq e^{s_n\eps  n}\p{W\hnn \geq \frac{c}{s_n}} \leq  e^{s_n\eps  n}\frac{s_n^{\alpha'}}{c^{\alpha'}}\E{(W\hnn )^{\alpha'}}.
    \end{align*}
     Plugging in $s_n = (a \log n)/n$, we get that for some constant $C'$, 
    \begin{align*}
        \E{e^{s_nW\hnn }\I{W\hnn \in (\frac{c}{s_n}, \eps  n]}}  & \leq s_n C' n^{a\eps }\left(\frac{a\log n}{n} \right)^{\alpha' - 1} = s_n C'a^{\alpha'-1} n^{a\eps  - \alpha' + 1}(\log n)^{\alpha' - 1}.
    \end{align*}
    For $0<a < (\alpha' - 1)/\eps $, we have $a\eps  - \alpha' + 1 < 0$ and thus $n^{a\eps  - \alpha' + 1}(\log n)^{\alpha' - 1}$ tends to zero as $n$ goes to infinity. It follows that 
    \[ \E{e^{s_nW\hnn }\I{W\hnn \in (\frac{c}{s_n}, \eps  n]}}  = o(s_n).\qedhere\]
\end{proof}

\begin{proof}[Proof of Remark~\ref{tail}]
We sketch the modifications needed to get the result stated in the remark. In place of Lemma~\ref{lem:sk_approximation} we use that
$$\p{n\sigma_1(n) \leq T_k \leq nk}  = (K'_{\rho}+o(1)) \, {f(n)}^{k},$$
which follows from a less subtle approximation. In the events $\mathcal E_1, \ldots, \mathcal E_3$  and $\mathcal E_4^j$ we use $I_n:=I_n(\rho, \infty)$ and here and in the further events we replace $\rho^{\subset}_2(n)$ by $k$. The analogue of Claim~1 then follows by the same arguments replacing $\rho_1$ by $\rho$ and $\rho_2$ by infinity.
The proofs of the analogues of 
Claim~2 and Claim~3 simplify as the small term $\rho_2(n)-\rho_1(n)$ is no longer needed on the right hand side. Finally,
the analogue of  Claim~4 follows by essentially the same arguments as in the proof of Theorem~\ref{thm:main_theorem}.
The combination of the claims yields the statement of Remark~\ref{tail}.
\end{proof}

\section{Outlook}

While it is a very natural model, our random graph is not one of the standard models of geometric random graphs. {It has been chosen for direct applicability of our general result. In an upcoming paper, we show that with some extra effort our approach is also suitable for a wide range of more classical geometric random graph models, for example 
the \emph{scale-free percolation} of~\cite{DHH13}  and the \emph{Boolean} or \emph{continuum percolation model} of~\cite{Hall85}. 
After the first version of our paper was completed,~\cite{SZ22} independently studied large deviations of the edge count in a non-spatial graph model and also observed multiple big jumps.}

In some classical models of geometric random graphs, for example the Boolean model, in which two vertices are connected by an edge if the associated balls intersect, points are randomly placed according to a Poisson process and this may contribute to the excess edges by moving points closer together. While for heavy tailed radius distributions we still believe that the large deviation behaviour follows the paradigm of the minimal number of jumps, such results will be harder to come by. 
If one assumes lighter tails of the radii, the large deviation behaviour of the number of edges becomes even harder to predict. \cite{10.1214/19-AOP1387} investigate the 
extreme case of the Boolean model with fixed deterministic radii. In this case
excess edges necessarily arise from the clumping of the points in the Poisson process. We also plan to investigate situations where both the points and the radii are random but the latter are light tailed.

\addtocontents{toc}{\protect\setcounter{tocdepth}{-1}}

\medskip


\medskip

\paragraph{\bf Acknowledgements}
We 
would like to thank 
Vitali Wachtel and Bert Zwart for alerting us to further interesting references, and  Remco van der Hofstad, Pim van der Hoorn, and Neeladri Maitra for useful discussions and for the permission to include the brief paragraph of the problem of integer values in Remark~\ref{integer} that originated in these discussions. This research was funded by DFG 
project 444092244 ``Condensation in random geometric graphs''
within the priority programme SPP~2265.

\addtocontents{toc}{\protect\setcounter{tocdepth}{2}}


\footnotesize

\bibliographystyle{plainnat}
\bibliography{citation}

%
%

\appendix

\end{document}